\documentclass{amsart}
\usepackage{amssymb}

\newcommand{\add}{\mathrm{add}}
\newcommand{\cof}{\mathrm{cof}}
\newcommand{\cov}{\mathrm{cov}}

\newcommand{\cf}{\mathrm{cf}}
\newcommand{\e}{\varepsilon}
\newcommand{\w}{\omega}
\newcommand{\E}{\mathcal E}

\newcommand{\Ra}{\Rightarrow}
\newcommand{\A}{\mathcal A}
\newcommand{\bgamma}{\overset{\leftrightarrow}\gamma}
\newcommand{\cbox}{\coprod}

\newcommand{\I}{\mathcal I}
\newcommand{\IR}{\mathbb R}

\newtheorem{theorem}{Theorem}[section]
\newtheorem{proposition}[theorem]{Proposition}
\newtheorem{lemma}[theorem]{Lemma}
\newtheorem{corollary}[theorem]{Corollary}
\theoremstyle{definition}
\newtheorem{remark}[theorem]{Remark}
\newtheorem{example}[theorem]{Example}

\title[Classifying homogeneous cellular ordinal balleans up to coarse equivalence]{Classifying homogeneous cellular ordinal balleans\\ up to coarse equivalence}
\author{Taras Banakh, Igor Protasov, Du\v san Repov\v s, and Sergii Slobodianiuk}
\address{T.Banakh: Ivan Franko National University of Lviv (Ukraine) and Jan Kochanowski University in Kielce (Poland)}
\email{t.o.banakh@gmail.com}
\address{I.Protasov, S.Slobodianiuk: Faculty of Cybernetics, Taras Shevchenko National University of Kyiv, Ukraine}
\email{i.v.protasov@gmail.com, slobodianiuk@yandex.ru}
\address{D.Repov\v s: Faculty of Education, and
Faculty of Mathematics and Physics,
University of Ljubljana,
Kardeljeva Pl.~16,
Ljubljana, Slovenia 1000}
\email{dusan.repovs@guest.arnes.si}

\begin{document}
\begin{abstract} For every ballean $X$ we introduce two cardinal characteristics $\cov^\flat(X)$ and $\cov^\sharp(X)$ describing the capacity of balls in $X$. We observe that these cardinal characteristics are invariant under coarse equivalence and prove that two cellular ordinal balleans $X,Y$ are coarsely equivalent if $\cof(X)=\cof(Y)$ and $\cov^\flat(X)=\cov^\sharp(X)=\cov^\flat(Y)=\cov^\sharp(Y)$. This result implies  that a cellular ordinal ballean $X$ is homogeneous if and only if $\cov^\flat(X)=\cov^\sharp(X)$. Moreover, two homogeneous cellular ordinal balleans $X,Y$ are coarsely equivalent if and only if $\cof(X)=\cof(Y)$ and $\cov^\sharp(X)=\cov^\sharp(Y)$ if and only if each of these balleans coarsely embeds into the other ballean. This means that the coarse structure of a  homogeneous cellular ordinal ballean $X$ is fully determined by the values of the cardinals $\cof(X)$ and  $\cov^\sharp(X)$. For every limit ordinal $\gamma$ we shall define a ballean $2^{<\gamma}$ (called the Cantor macro-cube), which in the class of cellular ordinal balleans of cofinality $\cf(\gamma)$ plays a role analogous to the role of the Cantor cube $2^{\kappa}$ in the class of zero-dimensional compact Hausdorff spaces. We shall also present a characterization of balleans which are coarsely equivalent to $2^{<\gamma}$. This characterization can be considered as an asymptotic analogue of Brouwer's characterization of the Cantor cube $2^\w$.
\end{abstract}

\maketitle

\section*{Introduction}

In this paper we study the structure of ordinal balleans, i.e., balleans that have well-ordered base of their coarse structure. Such balleans were introduced by Protasov in \cite{P03}. Some basic facts about ordinal balleans are discussed in Section~\ref{s:ord}. The main result of the paper is presented in Section~\ref{s:crit} containing a criterion for recognizing coarsely equivalent cellular ordinal balleans. In Section~\ref{s:homo} we shall use this criterion to classify homogeneous cellular ordinal balleans up to coarse equivalence. In Section~\ref{s:cantor} we apply this criterion to characterize balleans $2^{<\gamma}$ (called Cantor macro-cubes), which are universal objects in the class of cellular ordinal balleans.  In Section~\ref{s:cantor} also we identify the natural coarse structure on additively indecomposable ordinals.

\section{Ordinal Balleans}\label{s:ord}

The notion of a ballean was introduced by Protasov \cite{PB} as a large scale counterpart of a uniform space and is a modification of the notion of a coarse space introduced by Roe \cite{Roe}. Both notions are defined as sets endowed with certain families of entourages.

By an {\em entourage} on a set $X$ we understand any reflexive symmetric relation $\e\subset X\times X$. This means that $\e$ contains the diagonal $\Delta_X=\{(x,y)\in X\times X: x=y\}$ of $X\times X$ and is symmetric in the sense that $\e=\e^{-1}$ where $\e^{-1}=\{(y,x)\in X\times X: (x,y)\in \e\}$. An entourage $\e\subset X\times X$ will be called {\em cellular} if it is transitive, i.e., it is an equivalence relation on $X$.

Each entourage $\e\subset X\times X$ determines a cover $\{B(x,\e)\colon x\in X\}$ of $X$ by $\e$-balls $B(x,\e)=\{y\in X:(x,y)\in\e\}$. It follows that $\e=\bigcup_{x\in X}\{x\}\times B(x,\e)=\bigcup_{x\in X}B(x,\e)\times\{x\}$, so the entourage $\e$ can be fully recovered from the system of balls $\{B(x,\e):x\in\e\}$. For a subset $A\subset X$ we let $B(A,\e)=\bigcup_{a\in A}B(a,\e)$ denote the $\e$-neighborhood of $A$.

A {\em ballean} is a pair $(X,\mathcal E_X)$ consisting of a set $X$ and a family $\E_X$ of entourages on $X$ (called the {\em set of radii\/}) such that $\bigcup\E_X=X\times X$ and for any entourages $\e,\delta\in\E_X$ their composition $$\e\circ\delta=\{(x,z)\in X\times X:\exists y\in X\;(x,y)\in\e,\;(y,z)\in\delta\}$$ is contained in some entourage $\eta\in\E_X$. A ballean $(X,\E_X)$ is called a {\em coarse space} if the family $\mathcal E_X$ is closed under taking subentourages, i.e., for any $\e\in\E_X$ any entourage $\delta\subset\e$ belongs to $\E_X$.
In this case the set of radii $\E_X$ is called a {\em coarse structure} on $X$.
 Each set of radii $\E_X$ can be completes to the coarse structure ${\downarrow}\E_X$ consisting of all possible subentourages $\delta\subset \e\in\E_X$. In this case $\E_X$ is called a {\em base} of the coarse structure ${\downarrow}\E_X$. So, balleans can be considered as coarse spaces with a fixed base of their coarse structure. Coarse spaces and coarse structures were introduced by Roe \cite{Roe}.

Each subset $A\subset X$ of a ballean $(X,\E_X)$ carries the induced ballean structure $\E_A=\{\e\cap A^2:\e\in\E_X\}$. The ballean $(A,\E_A)$ will be called a {\em subballean} of $(X,\E_X)$.

Now we consider some examples of balleans.

\begin{example} Each metric space $(X,d)$ carries a canonical ballean structure $\E_X=\{\Delta_\e\}_{\e\in\IR_+}$ consisting of the entourages
$$\Delta_\e=\{(x,y)\in X\times X:d(x,y)\le \e\}$$parametrized by the set $\IR_+=[0,\infty)$ of non-negative real numbers. The ballean structure $\E_X=\{\Delta_\e\}_{\e\in\IR_+}$ generates the coarse structure ${\downarrow}\E_X$ consisting of all subentourages of the entourages $\Delta_\e$, $\e\in\IR_+$.
\end{example}

A ballean $X$ is called {\em metrizable} if its coarse structure is generated by a suitable metric.
Metrizable balleans belong to the class of ordinal balleans. A ballean $\mathbf X=(X,\E_X)$ is defined to be {\em ordinal} if its coarse structure ${\downarrow}\E_X$ has a well-ordered base $\mathcal B\subset\E_X$.
The latter means that $\mathcal B$ can be enumerated as $\{B_\alpha\}_{\alpha<\kappa}$ for some ordinal $\kappa$ such that $B_\alpha\subset B_\beta$ for all ordinals $\alpha<\beta<\kappa$. Passing to a cofinal subset of $\kappa$, we can always assume that $\kappa$ is a regular cardinal, equal of the cofinality $\cof(\mathbf X)$. By definition, the {\em cofinality} $\cof(\mathbf X)$ of a ballean $\mathbf X=(X,\E_X)$ is equal to the smallest cardinality of a base of the coarse structure ${\downarrow}\E_X$.

Ordinal balleans can be characterized as balleans $\mathbf X=(X,\E_X)$ whose cofinality equals the additivity number
$$
\add(\mathbf X)=\min\big\{|\A|:\A\subset{\downarrow}\E_X,\;\textstyle{\bigcup\A}\notin{\downarrow}\E_X\setminus\{X\times X\}\big\}.
$$

\begin{proposition} A ballean $X$ is ordinal if and only if $\cof(X)=\add(X)$.
\end{proposition}

\begin{proof} Assuming that a ballean $(X,\E_X)$ is ordinal, fix a well-ordered base $\{\e_\alpha\}_{\alpha<\kappa}$ of the coarse structure ${\downarrow}\E_X$ of $\E_X$. Passing to a cofinal subsequence, we can assume that $\kappa=\cf(\kappa)$ is a regular cardinal. If $\kappa=1$, then the ballean $(X,\E_X)$ is bounded and hence for the entourage $X\times X\in{\downarrow}\E_X$ the family $\A=\{X\times X\}$ has cardinality $|\A|=1$ and $\bigcup\A=X\times X\notin\E_X\setminus\{X\times X\}$. Therefore, $\add(X)=1=\cof(X)$. So, we assume that the regular cardinal $\kappa$ is infinite and hence $\e_\alpha\ne X\times X$ for all $\alpha<\kappa$.
Since $\add(X)\le\cof(X)$, it suffices to check that $\cof(X)\le \add(X)$. The definition of the cardinal $\cof(X)$ implies that $\cof(X)\le\kappa$. The inequality $\add(X)\ge\kappa\ge\cof(X)$ will follow as soon as we check that for any family $\A\subset\E_X$ of cardinality $|\A|<\kappa$ we get $\bigcup\A\in{\downarrow}\E_X\setminus\{X\times X\}$. For every set $A\in\A$ find an ordinal $\alpha_A<\kappa$ such that $A\subset\e_{\alpha_A}$. By the regularity of the cardinal $\kappa$, the cardinal $\beta=\sup\{\alpha_A:A\in\A\}$ is strictly smaller than $\kappa$. Consequently, $A\subset\e_{\alpha_A}\subset\e_\beta$ for every $A\in\A$ and hence $\bigcup\A\subset\e_\beta$ and $\bigcup\A\in{\downarrow}\E_X\setminus\{X\times X\}$. This completes the proof of the equality $\add(X)=\cof(X)$ for ordinal balleans.

Now we shall prove that a ballean $(X,\E_X)$ is ordinal if $\add(X)=\cof(X)$. Fix any base $\{\e_\alpha\}_{\alpha<\cof(X)}\subset {\downarrow}\E_X$ of the coarse structure ${\downarrow}\E_X$ of $X$. By definition of the additivity number $\add(X)$, for every $\alpha<\cof(X)=\add(X)$, the union $\tilde \e_\alpha=\bigcup_{\beta\le\alpha}\e_\beta$ belongs to the coarse structure ${\downarrow}\E_X$.
Then $(\tilde \e_\alpha)_{\alpha<\cof(X)}$ is a well-ordered base of the coarse structure ${\downarrow}\E_X$, which means that the ballean $(X,\E_X)$ is ordinal.
\end{proof}

An important property of ordinal balleans of uncountable cofinality is their cellularity. A ballean $(X,\E_X)$ is called {\em cellular} if its coarse structure ${\downarrow}\E_X$ has a base consisting of cellular entourages (i.e., equivalence relations). It can be shown that a ballean $(X,\E_X)$ is cellular if and only if for every $\e\in\E_X$ the cellular entourage $\e^{<\w}=\bigcup_{n\in\w}\e^n$ belongs to the coarse structure ${\downarrow}\E_X$. Here $\e^0=\Delta_X$ and $\e^{n+1}=\e^n\circ\e$ for all $n\in\w$. This characterization implies the following simple fact.

\begin{proposition} Each ordinal ballean $\mathbf X=(X,\E_X)$ with uncountable cofinality $\cof(\mathbf X)$ is cellular.
\end{proposition}

\begin{remark} By \cite[Theorem 3.1.3]{PZ}, a ballean $X$ is cellular if and only if it has asymptotic dimension $\mathrm{asdim}(X)=0$. A metrizable ballean $X$ is cellular if and only if its coarse structure is generated by an ultrametric (i.e., a metric $d$ satisfying the strong triangle inequality $d(x,z)\le\max\{d(x,y),d(y,z)\}$ for all point $x,y,z\in X$). More information on cellular balleans can be found in \cite[Chapter 3]{PZ}. For information on space of asymptotic dimension zero, see \cite{BDHM}.
\end{remark}

\begin{example} Every infinite cardinal $\kappa$ carries a natural ballean structure $\E_\kappa=\{\e_\alpha\}_{\alpha<\kappa}$ consisting of the entourages
$$\e_\alpha=\{(x,y)\in\kappa\times\kappa:x\le y+\alpha,\;\;y\le x+\alpha\}$$parametrized by ordinals $\alpha<\kappa$. The obtained ordinal ballean $(\kappa,\E_\kappa)$ will be denoted by $\overset{\leftrightarrow}{\kappa}$. Cardinal balleans $\overset{\leftrightarrow}\kappa$ were introduced in \cite{b5}. By Theorem 3 of \cite{b5}, the ballean $\overset{\leftrightarrow}{\kappa}$ is cellular for any uncountable cardinal $\kappa$.
\end{example}

\begin{example} Given an ordinal $\gamma$ and a transfinite sequence $(\kappa_\alpha)_{\alpha<\gamma}$ of non-zero cardinals, consider the ballean $$\cbox_{\alpha\in\gamma}\kappa_\alpha=\big\{(x_\alpha)_{\alpha\in\gamma}\in
\prod_{\alpha\in\gamma}\kappa_\alpha:|\{\alpha\in\gamma:x_\alpha\ne 0\}|<\w\big\}$$endowed with the ballean structure $\{\e_\beta\}_{\beta<\gamma}$ consisting of the entourages $$\e_\beta=\big\{\big((x_\alpha)_{\alpha\in\gamma},(y_\alpha)_{\alpha\in\gamma}\big)\in\Big(\cbox_{\alpha<\gamma}\kappa_\alpha\Big)^2:\forall \alpha>\beta\;\;(x_\alpha=y_\alpha)\big\}\mbox{ \ for \ }\beta<\gamma.$$
The ballean $\cbox_{\alpha\in\gamma}\kappa_\alpha$ is called the {\em asymptotic product} of cardinals $\kappa_\alpha$, $\alpha\in\gamma$. It is a cellular ordinal ballean whose cofinality equals $\cf(\gamma)$, the cofinality of the ordinal $\gamma$.

If all cardinals $\kappa_\alpha$, $\alpha\in\gamma$, are equal to a fixed cardinal $\kappa$, then the asymptotic product $\cbox_{\alpha\in\gamma}\kappa_\alpha$ will be denoted by $\kappa^{<\gamma}$. For a limit ordinal $\gamma$ the ballean $2^{<\gamma}$ is called a {\em Cantor macro-cube}. The Cantor macro-cube $2^{<\w}$ was characterized in \cite{BZ}. This characterization will be extended to all Cantor macro-cubes in Theorem~\ref{char}.
\end{example}

Balleans are objects of the (coarse) category whose morphisms are coarse maps. A map $f\colon X\to Y$ between two balleans $(X,\E_X)$ and $(Y,\E_Y)$ is called {\em coarse} if for each $\e\in\E_X$ there is $\delta\in\E_Y$ such that $\{(f(x),f(y)):(x,y)\in\e\}\subset\delta$. A map $f\colon X\to Y$ is called a {\em coarse isomorphism} if $f$ is bijective and both maps $f$ and $f^{-1}$ are coarse. In this case the balleans $(X,\E_X)$ and $(Y,\E_Y)$ are called {\em coarsely isomorphic}.
It follows that each ballean $(X,\E_X)$ is coarsely isomorphic to the coarse space $(X,{\downarrow}\E_X)$.

Coarse isomorphisms play the role of isomorphisms in the coarse category (whose objects are balleans and morphisms are coarse maps). Probably a more important notion is that of a coarse equivalence of balleans. Two balleans $(X,\E_X)$ and $(Y,\E_Y)$ are {\em coarsely equivalent} if they contain coarsely isomorphic large subspaces $L_X\subset X$ and $L_Y\subset Y$. A subset $L$ of a ballean $(X,\E_X)$ is called {\em large} if $X=B(L,\e)$ for some entourage $\e\in \E_X$.

Coarse equivalences can be alternatively defined using multi-valued maps.
By a multi-valued map (briefly, a {\em multi-map}) $\Phi\colon X\multimap Y$ between two sets $X,Y$ we
understand any subset $\Phi\subset X\times Y$. For a subset $A\subset X$ by $\Phi(A)=\{y\in Y:\exists a\in A\mbox{ with }(a,y)\in\Phi\}$ we denote the image of $A$ under the
multi-map $\Phi$. Given a point $x\in X$ we write $\Phi(x)$
instead of $\Phi(\{x\})$.

The inverse $\Phi^{-1}\colon Y\multimap X$ of the multi-map $\Phi\colon X\multimap Y$ is the
multi-map $$\Phi^{-1}=\{(y,x)\in Y\times X: (x,y)\in\Phi\}\subset
Y\times X$$ assigning to each point $y\in Y$ the set $\Phi^{-1}(y)=\{x\in X:y\in\Phi(x)\}$. For two multi-maps $\Phi\colon X\multimap Y$ and $\Psi\colon Y\multimap Z$ we
define their composition $\Psi\circ\Phi\colon X\multimap Z$ as usual:
$$\Psi\circ\Phi=\{(x,z)\in X\times Z:\exists y\in Y\mbox{ such that $(x,y)\in \Phi$ and $(y,z)\in\Psi$}\}.$$

A multi-map $\Phi\colon X\multimap Y$ between two balleans $(X,\E_X)$ and $(Y,\E_Y)$ is called {\em coarse} if for every $\e\in\E_X$ there is an entourage $\delta\in\E_Y$ containing the set $\w_\Phi(\e)=\bigcup_{(x,y)\in\e}\Phi(x)\times\Phi(y)$ called the {\em $\e$-oscillation} of $\Phi$.
More precisely, for a function $\varphi\colon \E_X\to\E_Y$ a multi-map $\Phi\colon X\multimap Y$ is defined to be {\em $\varphi$-coarse} if $\w_\Phi(\e)\subset\varphi(\e)$ for every $\e\in\E_X$. So, a multi-map $\Phi\colon X\multimap Y$ is coarse if and only if $\Phi$ is $\varphi$-coarse for some $\varphi\colon \E_X\to\E_Y$.
 It follows that a (single-valued) map $f\colon X\to Y$ is coarse if and only if it is coarse as a multi-map.

A multi-map $\Phi\colon X\multimap Y$ between two balleans is called a {\em coarse embedding} if  $\Phi^{-1}(Y)=X$ and both maps $\Phi$ and $\Phi^{-1}$ are coarse. If, in addition,  $\Phi(X)=Y$, then the multi-map $\Phi\colon X\multimap Y$ is called a {\em coarse equivalence} between the balleans $X$ and $Y$.   By analogy with the proof of Proposition~2.1 \cite{BZ}, it can be shown that two balleans $X,Y$ are coarsely equivalent if and only if there is a coarse equivalence $\Phi\colon X\multimap Y$.

The study of balleans (or coarse spaces) up to their coarse equivalence is one of principal tasks of Coarse Geometry \cite{BS}, \cite{Nowak},  \cite{PZ}, \cite{Roe}.

\begin{example} Let $G$ be a group. An ideal $\mathcal I$ in the Boolean algebra of all subsets of $G$ is called a {\em group ideal} if $G=\bigcup \mathcal I$ and if for any $A,B\in \mathcal I$ we get $AB^{-1}\in \mathcal I$.

Let $\I$ be a group ideal $\I$ on a group $G$ and $X$ be a transitive $G$-space endowed with an action $G\times X\to X$ of the group $G$. The $G$-space $X$ carries the ballean structure $\E_{X,G,\I}=\{\e_A\}_{A\in\I}$ consisting of the entourages $\e_A=\{(x,y)\in X:x\in (A\cup\{1_G\}\cup A^{-1})\cdot y\}$ parametrized by sets $A\in\I$. Here by $1_G$ we denote the unit of the group $G$.

By Theorems 1 and 3 of \cite{b4}, every (cellular) ballean $(X,\E_X)$ is coarsely isomorphic to the ballean $(X,\E_{X,G,\I})$ for a suitable group $G$ of permutations of $X$ and a suitable group ideal $\I$ of $G$ (having a base consisting of subgroups of $G$).
\end{example}

\begin{example} Let $G$ be a group endowed with the ballean $\E_G$ consisting of entourages $\e_F=\{(x,y)\in G\times G:xy^{-1}\in F\}$ parametrized by finite subsets $F=F^{-1}\subset G$ containing the unit $1_G$ of the group. By \cite[9.8]{PB} the ballean $(G,\E_G)$ is cellular if and only if the group $G$ is locally finite (in the sense that each finite subset of $G$ is contained in a finite subgroup of $G$). By \cite{BZ}, any two infinite countable locally finite groups $G,H$ are coarsely equivalent. On the other hand, by \cite{P02}, two countable locally finite groups $G,H$ are coarsely isomorphic if and only if $\phi_G=\phi_H$. Here for a group $G$ by $\phi_G\colon \Pi\to\w\cup\{\w\}$ we denote its {\em factorizing function}. It is defined on the set $\Pi$ of prime numbers and assigns to each prime number $p\in\Pi$ the (finite or infinite) number $$\phi_G(p)=\sup\{k\in\w: G\mbox{ contains a subgroup of cardinality }p^k\}.$$
\end{example}

\section{A criterion for a coarse equivalence of two cellular ordinal balleans}\label{s:crit}

In this section we introduce two cardinal characteristics called covering numbers of a ballean, and using these cardinal characteristics give a criterion for a coarse equivalence of two cellular ordinal balleans.

Given a subset $A\subset X$ of a set $X$ and an entourage $\e\subset X\times X$
consider the cardinal
$$\cov_\e(A)=\min\{|C|:C\subset X,\;A\subset B(C,\e)\}$$equal to the smallest number of $\e$-balls covering the set $A$.

For every ballean $(X,\E_X)$ consider the following cardinals:
\begin{itemize}
\item $\cov^\sharp(X,\E_X)$, equal to the smallest cardinal $\kappa$ for which there is an entourage $\e\in\E_X$ such that for every $\delta\in\E_X$ we get $\sup_{x\in X}\cov_\e(B(x,\delta))<\kappa$;
\item $\cov^\flat(X,\E_X)$, equal to the largest cardinal $\kappa$ such that for any cardinal $\lambda<\kappa$ and entourage $\e\in\E_X$ there is $\delta\in\E_X$ such that $\min_{x\in X}\cov_\e(B(x,\delta))\ge\lambda$.
\end{itemize}
It follows that
$$\cov^\sharp(X)=\min_{\e\in\E_X}\sup_{\delta\in\E_X}\big(\sup_{x\in X}\cov_\e(B(x,\delta))\big)^+$$and
$$\cov^\flat(X)=\min_{\e\in\E_X}\sup_{\delta\in\E_X}\big(\min_{x\in X}\cov_\e(B(x,\delta))\big)^+,$$
where $\kappa^+$ denotes the smallest cardinal which is larger than $\kappa$.
Cardinals are identified with the smallest ordinals of  given cardinality.

The following proposition can be proved by analogy with the proof of Lemmas~3.1 and 3.2 in \cite{BR}.

\begin{proposition} If a ballean $X$ coarsely embeds into a ballean $Y$, then $\cov^\sharp(X)\le\cov^\sharp(Y)$. If balleans $X,Y$ are coarsely equivalent, then $\cov^\flat(X)=\cov^\flat(Y)$ and $\cov^\sharp(X)=\cov^\sharp(Y)$.
\end{proposition}

Observe that the inequality $\cov^\sharp(X)\le\w$ means that $X$ has bounded geometry while $\cov^\flat(X)\ge\w$ means that $X$ has no isolated balls (see \cite{Zar}). By \cite{BZ}, any two metrizable cellular balleans of bounded geometry and without isolated balls are coarsely equivalent. In \cite{BR} this result was extended to the following criterion: two metrizable cellular balleans $X,Y$ are coarsely equivalent if $\cov^\flat(X)=\cov^\sharp(X)=\cov^\sharp(Y)=\cov^\flat(Y)$. In this paper we further extend this criterion to cellular ordinal balleans and prove the following main result of this paper.

\begin{theorem}\label{main} Let $X,Y$ be any two cellular ordinal balleans with $\cof(X)=\cof(Y)$.
\begin{enumerate}
\item If $\cov^\sharp(X)\le\cov^\flat(Y)$, then $X$ is coarsely equivalent to a subspace of $Y$.
\item If $\cov^\flat(X)=\cov^\sharp(X)=\cov^\sharp(Y)=\cov^\flat(Y)$, then the balleans $X$ and $Y$ are coarsely equivalent.
\end{enumerate}
\end{theorem}

The proof of this theorem will be presented in Section~\ref{s:pfmain}. First we shall discuss some applications of this theorem.

\section{Classifying homogeneous cellular ordinal balleans}\label{s:homo}

In this section we shall apply Theorem~\ref{main} to show that for a cellular ordinal ballean $X$ the equality $\cov^\flat(X)=\cov^\sharp(X)$ is equivalent to the homogeneity of $X$, defined as follows.

A ballean $(X,\E_X)$ is called {\em homogeneous} if there is a function $\varphi\colon \E_X\to\E_X$ such that for every points $x,y\in X$ there is a coarse equivalence $\Phi\colon X\multimap X$ such that $y\in\Phi(x)$ and the multi-maps $\Phi$ and $\Phi^{-1}$ are $\varphi$-coarse. Let us recall that $\Phi\colon X\multimap X$ is $\varphi$-coarse if $\w_\Phi(\e):=\bigcup_{(x,y)\in\e}\Phi(x)\times\Phi(x)\subset\varphi(\e)$ for every $\e\in\E_X$.

The following proposition shows that homogeneity is preserved by coarse equivalences.

\begin{proposition}\label{hompres} A ballean $X$ is homogeneous if and only if it is coarsely equivalent to a homogeneous ballean $Y$.
\end{proposition}

\begin{proof} The ``only if'' part is trivial. To prove the ``if'' part, assume that a ballean $(X,\E_X)$ admits a coarse equivalence $\Phi\colon X\multimap Y$ with a homogeneous ballean $(Y,\E_Y)$. By the homogeneity of $(Y,\E_Y)$, there is a function $\varphi_Y\colon \E_Y\to\E_Y$ such that for any points $y,y'\in Y$ there is a coarse equivalence $\Psi\colon Y\multimap Y$ such that $y'\in \Psi(y)$ and both multi-maps $\Psi$ and $\Psi^{-1}$ are $\varphi_Y$-coarse. Since $\Phi$ is a coarse equivalence, there are functions $\varphi_{X,Y}:\E_X\to\E_Y$ and $\varphi_{Y,X}:\E_Y\to\E_X$ such that $\w_\Phi(\e)\subset\varphi_{X,Y}(\e)$ and $\w_{\Phi^{-1}}(\delta)\subset\varphi_{Y,X}(\delta)$ for every $\e\in\E_X$ and $\delta\in\E_Y$. We claim that the function $$\varphi_X=\varphi_{Y,X}\circ\varphi_Y\circ\varphi_{X,Y}:\E_X\to\E_X$$
witnesses that the ballean $X$ is homogeneous. Indeed, given any points $x,x'$, we can choose points $y\in\Phi(x)$, $y'\in\Phi(x')$ and find a coarse equivalence $\Psi_Y\colon Y\to Y$ such that $y'\in\Psi_Y(y)$ and the multi-maps $\Psi_Y$ and $\Psi_Y^{-1}$ are $\varphi_Y$-coarse. It can be shown that the multi-map $\Psi_X=\Phi^{-1}\circ\Psi_Y\circ\Phi\colon X\multimap X$ has the desired properties: $x'\in \Phi^{-1}(y')\subset \Phi^{-1}(\Psi_Y(y))\subset\Phi^{-1}(\Psi_Y(\Phi(x)))=\Phi_X(x)$ and $\w_{\Phi_X}(\e)\cup\w_{\Phi_X^{-1}}(\e)\subset \varphi_X(\e)$ for all $\e\in\E_X$.
\end{proof}

\begin{proposition}\label{homocov} If a ballean $X$ is homogeneous, then $\cov^\flat(X)=\cov^\sharp(X)$.
\end{proposition}

\begin{proof} Since $\cov^\flat(X)\le \cov^\sharp(X)$, it suffices to check that
$\cov^\flat(X)\ge\cov^\sharp(X)$. This inequality will follow as soon as given $\e\in\E_X$ and a cardinal $\kappa<\cov^\sharp(X)$, we find an entourage $\delta\in\E_X$ such that $\min_{x\in X}\cov_\e(B(x,\delta))\ge\kappa$. By the homogeneity of $X$, there is a function $\varphi\colon \E_X\to\E_X$ such that for any points $x,y\in X$ there is a coarse equivalence $\Phi\colon X\multimap X$ such that  $y\in\Phi(x)$ and the multi-maps $\Phi$ and $\Phi^{-1}$ are $\varphi$-coarse.

By the definition of the cardinal $\cov^\sharp(X)>\kappa$, for the entourage $\e'=\varphi(\e)$, there is an entourage $\delta'\in\E_X$ and a point $x'\in X$ such that $\cov_{\e'}(B(x',\delta'))\ge\kappa$. We claim that the entourage $\delta=\varphi(\delta')\in\E_X$ has the required property: $\min_{x\in X}\cov_{\e}(B(x,\delta))\ge\kappa$. Assume conversely that $\cov_\e(B(x,\delta))<\kappa$ for some $x\in X$. By the homogeneity of $X$ and the choice of $\varphi$, there is a coarse equivalence $\Phi\colon X\multimap X$ such that $x'\in\Phi(x)$ and the multi-maps $\Phi$ and $\Phi^{-1}$ are $\varphi$-coarse.
Since $\cov_\e(B(x,\delta))<\kappa$, there is a subset $C\subset X$ of cardinality $|C|<\kappa$ such that $B(x,\delta)\subset \bigcup_{c\in C}B(c,\e)$. For every $c\in C$ fix a point $y_c\in \Phi(c)$ and observe that for every point $b\in B(c,\e)$ we get $(b,c)\in\e$ and hence $\Phi(b)\times\Phi(c)\subset\w_\Phi(\e)\subset\varphi(\e)$ and $\Phi(b)\subset B(y_c,\varphi(\e))=B(y_c,\e')$, which implies $\Phi(B(c,\e))\subset B(y_c,\e')$. Taking into account that $B(x,\delta)\subset \bigcup_{c\in C}B(c,\e)$, we get $\Phi(B(x,\delta))\subset \bigcup_{c\in C}\Phi(B(c,\e))\subset\bigcup_{c\in C}B(y_c,\e')$, which implies $\cov_{\e'}(\Phi(B(x,\delta)))\le|C|<\kappa$.

We claim that $B(x',\delta')\subset\Phi(B(x,\delta))$. Indeed, for any point $y'\in B(x',\delta')$ we can fix a point $y\in\Phi^{-1}(x')$ and observe that $(y',x')\in\delta'$ implies $(y,x)\in\Phi^{-1}(y')\times\Phi^{-1}(x')\subset\w_{\Phi^{-1}}(\delta')\subset\varphi(\delta')=\delta$ (by the $\varphi$-coarse property of the multi map $\Phi^{-1}$). Then $y\in B(x,\delta)$ and $y'\in \Phi(y)\subset \Phi(B(x,\delta))$. Finally, we get $B(x',\delta')\subset \Phi(B(x,\delta))$ and $\cov_{\e'}(B(x',\delta'))\le\cov_{\e'}(\Phi(B(x,\delta)))\le|C|<\kappa$, which contradicts the choice of $\delta'$ and  $x'$. This completes the proof of the equality $\cov^\flat(X)=\cov^\sharp(X)$.
\end{proof}

\begin{theorem}\label{homo} A cellular ordinal ballean $X$ is homogeneous if and only if $\cov^\flat(X)=\cov^\sharp(X)$.
\end{theorem}

\begin{proof} The ``only if'' part follows from Proposition~\ref{homocov}. To prove the ``if'' part, assume $X$ is a cellular ordinal ballean with $\cov^\flat(X)=\cov^\sharp(X)$. Let $\gamma=\cof(X)=\add(X)$. The definition of the cardinal  $\kappa=\cov^\flat(X)=\cov^\sharp(X)$ implies that there exists a non-decreasing transfinite sequence of cardinals $(\kappa_\alpha)_{\alpha<\gamma}$ such that $\kappa=\sup_{\alpha<\gamma}\kappa_\alpha^+$.
Choose an increasing transfinite sequence of groups $(G_\alpha)_{\alpha<\gamma}$ such that $G_\alpha=\bigcup_{\beta<\alpha}G_\beta$ for every limit ordinal $\alpha<\gamma$ and $|G_{\alpha+1}/G_\alpha|=\kappa_\alpha$ for every ordinal $\alpha<\gamma$.

Consider the group $G=\bigcup_{\alpha<\gamma}G_\alpha$ endowed with the ballean structure $\E_G=(\e_\alpha)_{\alpha<\gamma}$ consisting of the entourages
$$\e_\alpha=\{(x,y)\in G:x^{-1}y\in G_\alpha\}\mbox{ for $\alpha<\gamma$}.$$
It is clear that the left shifts are $\mathrm{id}$-coarse isomorphisms of $(G,\E_G)$, which implies that the ballean $(G,\E_G)$ is homogeneous. It is clear that $\add(G,\E_G)=\cof(G,\E_G)=\gamma$ and $$\cov^\flat(G,\E_G)=\cov^\sharp(G,\E_G)=\min_{\alpha<\gamma}\sup_{\alpha\le\beta<\gamma}|G_\beta/G_\alpha|^+=\sup_{\alpha<\gamma}\kappa_\alpha^+=\kappa.$$
Applying Theorem~\ref{main}, we conclude that $X$ is coarsely equivalent to the homogeneous ballean $(G,\E_G)$ and hence $X$ is homogeneous according to Proposition~\ref{hompres}.
\end{proof}

The following corollary of Theorems~\ref{main} and \ref{homo} shows that the cardinals $\cof(X)$ and $\cov^\sharp(X)$ fully determine the coarse structure of a homogeneous cellular ordinal ballean $X$.

\begin{theorem} For any two homogeneous cellular ordinal balleans $X,Y$ the following conditions are equivalent:
\begin{enumerate}
\item $X$ and $Y$ are coarsely equivalent;
\item $X$ is coarsely equivalent to a subspace of $Y$ and $Y$ is coarsely equivalent to a subspace of $X$;
\item $\cof(X)=\cof(Y)$ and $\cov^\sharp(X)=\cov^\sharp(Y)$.
\end{enumerate}
\end{theorem}

\begin{proof} The implication $(1)\Ra(2)$ is trivial, the implication $(2)\Ra(3)$ follows by the invariance of the cardinal characteristics $\cof$ and $\cov^\sharp$ under coarse equivalence and their monotonicity under taking subspaces, and the final implication $(3)\Ra(1)$ follows from Theorems~\ref{main} and \ref{homo}.
\end{proof}

\section{Recognizing the coarse structure of Cantor macro-cubes and cardinal balleans}\label{s:cantor}

It is easy to see that for any ordinal $\gamma$ and transfinite sequence $(\kappa_\alpha)_{\alpha\in \gamma}$ of non-zero cardinals the asymptotic product $\cbox_{\alpha\in\gamma}\kappa_\alpha$ is a homogeneous cellular ordinal ballean whose cofinality equals $\cf(\gamma)$, the cofinality of the ordinal $\gamma$. In particular, the Cantor macro-cube $2^{<\gamma}$ is a homogeneous cellular ordinal ballean of cofinality $\cof(2^{<\gamma})=\cf(\gamma)$.

To evaluate the covering numbers of $2^{<\gamma}$, for an ordinal $\gamma$, consider the ordinal
$$ \lfloor\gamma\rfloor=\min\{\alpha:\gamma=\beta+\alpha\mbox{ for some $\beta<\gamma$}\}$$called the {\em tail} of $\gamma$, and the cardinal
$$\lceil\gamma\rceil=\min\{\alpha:\gamma\le\beta+|\alpha|\mbox{ for some $\beta<\gamma$}\}
$$called the {\em cardinal tail} of $\gamma$. It is clear that $\lfloor\gamma\rfloor\le\lceil\gamma\rceil$. Moreover,
$$\lceil\gamma\rceil=
\begin{cases}
|\lfloor\gamma\rfloor|&\mbox{if $\lfloor\gamma\rfloor$ is a cardinal},\\
|\lfloor\gamma\rfloor|^+&\mbox{otherwise}.
\end{cases}
$$

The equality $\gamma=\lfloor\gamma\rfloor$ holds if and only if the ordinal $\gamma$ is {\em additively indecomposable}, which means that $\alpha+\beta<\gamma$ for any ordinals $\alpha,\beta<\gamma$.

The following proposition can be derived from the definition of $2^{<\gamma}$.

\begin{proposition}\label{cantor} For every ordinal $\gamma$ the Cantor macro-cube $2^{<\gamma}$ is a cellular ordinal ballean with $$\add(2^{<\gamma})=\cof(2^{<\gamma})=\cf(\gamma)\mbox{ \ and \  }\cov^\flat(2^{<\gamma})=\cov^\sharp(2^{<\gamma})=\lceil\gamma\rceil.$$
\end{proposition}

The following theorem (which can be derived from Proposition~\ref{cantor} and Theorem~\ref{main}) shows that in the class of cellular ordinal balleans, the Cantor macro-cubes $2^{<\gamma}$ play a role analogous to the role of the Cantor cubes $2^\kappa$ in the class of zero-dimensional compact Hausdorff spaces.

\begin{theorem}\label{cant} Let $\gamma$ be any ordinal and $X$ be any cellular ordinal ballean such that $\cof(X)=\cf(\gamma)$.
\begin{enumerate}
\item If $\lceil\gamma\rceil\le \cov^\flat(X)$, then $2^{<\gamma}$ is coarsely equivalent to a subspace of $X$;
\item If $\cov^\sharp(X)\le\lceil\gamma\rceil$, then $X$ is coarsely equivalent to a subspace of $2^{<\gamma}$.
\item If $\cov^\flat(X)=\cov^\sharp(X)=\lceil\gamma\rceil$, then $X$ is coarsely equivalent to $2^{<\gamma}$.
\end{enumerate}
\end{theorem}

Proposition~\ref{cantor} and Theorem~\ref{cant} imply the following characterization of the Cantor macro-cube $2^{<\gamma}$ which extends the characterization of the Cantor macro-cube $2^{<\w}$ proved in \cite{BZ}.

\begin{theorem}\label{char} For any ordinal $\gamma$ and any ballean $X$ the following conditions are equivalent:
\begin{enumerate}
\item $X$ is coarsely equivalent to $2^{<\gamma}$;
\item $X$ is cellular, $\add(X)=\cof(X)=\cf(\gamma)$ and $\cov^\flat(X)=\cov^\sharp(X)=\lceil\gamma\rceil$.
\end{enumerate}
\end{theorem}

\begin{corollary}
For any two ordinals $\beta,\gamma$ the Cantor macro-cubes $2^{<\beta}$ and $2^{<\gamma}$ are coarsely equivalent of and only if $\cf(\beta)=\cf(\gamma)$ and $\lceil\beta\rceil=\lceil\gamma\rceil$.
\end{corollary}

Finally, we recognize the coarse structure on the ballean $\overset{\leftrightarrow}\gamma$ supported by an additively indecomposable ordinal $\gamma$.
Given any non-zero ordinal $\gamma$ we consider the family $\{\e_\alpha\}_{\alpha<\gamma}$ of the entourages
$$\e_\alpha=\{(x,y)\in\gamma\times\gamma:x\le y+\alpha\mbox{ and }\;y\le x+\alpha\}$$
for $\alpha<\gamma$. It is easy to see that $(\gamma,\{\e_\alpha\})_{\alpha<\gamma}$ is a ballean if and only if the ordinal $\gamma$ is additively indecomposable (which means that $\alpha+\beta<\gamma$ for any ordinals $\alpha,\beta<\gamma$). This ballean will be denoted by $\bgamma$.

The following theorem classifies the balleans $\bgamma$ up to coarse equivalence.

\begin{theorem} For any additively indecomposable ordinal $\gamma$ the ballean $\bgamma$ is coarsely equivalent to:
\begin{itemize}
\item $\overset{\leftrightarrow}\w$ if and only if $\gamma=\beta\cdot \w$ for some $\beta$;
\item $2^{<\gamma}$, otherwise.
\end{itemize}
\end{theorem}

\begin{proof} If $\gamma=\beta\cdot \w$ for some ordinal $\beta$, then the ballean $\bgamma$ is coarsely equivalent to $\overset{\leftrightarrow}\w$ since $\bgamma$ contains the large subset $L=\{\beta\cdot n:n\in\w\}$, which is coarsely isomorphic to $\overset{\leftrightarrow}\w$.

Now assume that $\gamma\ne\beta\cdot\w$ for any ordinal $\beta$. Since $\gamma$ is additively indecomposable, this means that $\beta\cdot\w<\gamma$ for any ordinal $\beta<\gamma$, which implies that the ballean $\bgamma$ is cellular. Since $\add(\bgamma)=\cof(\bgamma)=\cf(\gamma)$ and $\cov^\flat(\bgamma)=\cov^\sharp(\bgamma)=\lceil\gamma\rceil$, the cellular ordinal ballean $\bgamma$ is coarsely equivalent to $2^{<\gamma}$ according to Theorem~\ref{char}.
\end{proof}

\begin{remark} Let us observe that for any ordinal $\gamma$ the balleans $2^{<\gamma}$ and $\overset{\leftrightarrow}\w$ are not coarsely equivalent since the ballean $2^{<\gamma}$ is cellular whereas $\overset{\leftrightarrow}\w$ is not.
\end{remark}

\section{Embedding cellular ordinal balleans into asymptotic products of cardinals}

In this section we construct coarse embeddings of cellular ordinal balleans into asymptotic products of cardinals. This embedding will play a crucial role in the proof of Theorem~\ref{main} presented in the next section.

Let us observe that for any transfinite sequence of cardinals $(\kappa_\alpha)_{\alpha<\gamma}$  the asymptotic product $\cbox_{\alpha<\gamma}\kappa_\alpha$ carries an operation of coordinatewise addition of sequences induces by the operation of addition of ordinals.  For ordinals $\beta<\gamma$ and $y\in\kappa_\alpha$ let
$y\cdot\delta_\beta$ denote the sequence $(x_\alpha)_{\alpha<\gamma}\in \cbox_{\alpha<\gamma}\kappa_\alpha$ such that $x_\alpha=y$ if $\alpha=\beta$ and $x_\alpha=0$ otherwise. It follows that each element $(x_\alpha)_{\alpha<\gamma}\in\cbox_{\alpha<\gamma}\kappa_\alpha$ can be written as $\sum_{\alpha\in A}x_\alpha\cdot\delta_\alpha$ for the finite set $A=\{\alpha<\gamma:x_\alpha\ne 0\}$.

The following lemma exploits and develops the decomposition technique used in \cite{P02}, \cite[\S10]{PB} and \cite{PT}.

\begin{lemma}\label{key} Let $X$ be an ordinal ballean of infinite cofinality $\gamma$ and $(\e_\alpha)_{\alpha<\gamma}$ be a well-ordered base of the coarse structure of $X$ consisting of cellular entourages such that $\e_\beta=\bigcup_{\alpha<\beta}\e_\alpha$ for all limit ordinals $\beta<\gamma$. For every $\alpha<\gamma$ and $x\in X$ let $\kappa_\alpha(x)=\cov_{\e_\alpha}(B(x,\e_{\alpha+1}))$ and put
$\kappa_\alpha=\min_{x\in X}\kappa_\alpha(x)$ and $\bar\kappa_\alpha=
\sup_{x\in X}\kappa_\alpha(x)$.
Then the ballean $X$ is coarsely equivalent to a subbalean $Y\subset\cbox_{\alpha<\gamma}\bar\kappa_\alpha$ containing the set $\cbox_{\alpha<\lambda}\kappa_\alpha$.
\end{lemma}

\begin{proof} For each two points $x,y\in X$ let
$$d(x,y)=\min\{\alpha<\gamma:(x,y)\in\e_\alpha\}$$ and observe that for any pair $(x,y)\notin\e_0$ the ordinal $d(x,y)$ is not limit (as $\e_\beta=\bigcup_{\alpha<\beta}\e_\alpha$ for any limit ordinal $\beta<\gamma$). Consequently we can find an ordinal $d^-(x,y)$ such that $d(x,y)=d^-(x,y)+1$.

Fix any well-ordering $\preceq$ of the set $X$. Given a non-empty subset $B\subset X$ denote by $\min B$  the smallest point of $B$ with respect to the well-order $\preceq$ and for every $\alpha<\gamma$ let $c_\alpha\colon X\to X$ be the map assigning to each point $x\in X$ the smallest element $c_\alpha(x)=\min B(x,\e_\alpha)$ of the ball $B(x,\e_\alpha)$. Since $\e_\alpha$ is an equivalence relation,
 $B(x,\e_\alpha)=B(c_\alpha(x),\e_\alpha)$. To simplify the notation in the sequel we shall denote the ball $B(x,\e_\alpha)$ by $B_\alpha(x)$.

Observe that for every $\alpha<\gamma$ and ball $B\in\{B_{\alpha+1}(x):x\in X\}$ the set $c_\alpha(X)\cap B$ has cardinality $\kappa_\alpha(\min B)$, so we can fix a map $n_{\alpha,B}:B\to \kappa_\alpha(\min B)$ such that $\{B_\alpha(x):x\in B\}=\{n^{-1}_{\alpha,B}(\beta):\beta\in\kappa_\alpha(\min B)\}$ and $n^{-1}_{\alpha,B}(0)=B_\alpha(\min B)$. Finally, define a map $n_\alpha\colon X\to \bar\kappa_\alpha$ assigning to each point $y\in X$ the number $n_\alpha(y):=n_{\alpha,B_{\alpha+1}(y)}(y)$ of the $\e_\alpha$-ball containing $y$ in the partition of the $\e_{\alpha+1}$-ball $B_{\alpha+1}(y)$.
The definition of the cardinal $\kappa_\alpha$ implies that $\kappa_\alpha\subset \kappa_\alpha(x)= n_\alpha(B_{\alpha+1}(x))$ for every $x\in X$.

For every $x\in X$ define a map $f_x\colon X\to \cbox_{\alpha<\gamma}\bar\kappa_\alpha$ by the recursive formula
$$
f_x(y)=\begin{cases}
0&\mbox{ if $d(x,y)=0$};\\
f_{\min B_{d^-(x,y)}(y)}(y)+n_{d^-(x,y)}(y)\cdot\delta_{d^-(x,y)}&\mbox{otherwise}.
\end{cases}
$$
Since $d(\min B_{d^-(x,y)}(y),y)<d(x,y)$ the function $f_x$ is well-defined.

It can be shown that for every $x\in X$ the function $f_x\colon X\to \cbox_{\alpha<\gamma}\bar\kappa_\alpha$
determines a coarse equivalence of $X$ with the subspace $f(X)$ of  $\cbox_{\alpha<\gamma}\bar\kappa_\alpha$ containing the asymptotic product $\cbox_{\alpha<\gamma}\kappa_\alpha$.
\end{proof}

\section{Proof of Theorem~\ref{main}}\label{s:pfmain}

 Assume that $X,Y$ are two cellular balleans with $\gamma=\add(X)=\cof(X)=\cof(Y)=\add(X)$ and $\kappa=\cov^\flat(X)=\cov^\sharp(X)=\cov^\sharp(Y)=\cov^\flat(Y)$ for some cardinals $\gamma$ and $\kappa$. Let $\E_X$, $\E_Y$ denote the ballean structures of $X$ and $Y$, respectively.

Separately we shall consider 4 cases.
\smallskip

1) $\gamma=0$. In this case the balleans  $X,Y$ are empty and hence coarsely equivalent.
\smallskip

2) $\gamma=1$. In this case the balleans $X,Y$ are bounded and hence are coarsely equivalent.
\smallskip

3) $\gamma=\w$. Since $X$ is a cellular ballean with $\cof(X)=\gamma=\w$, the coarse structure ${\downarrow}\E_X$ of $X$ has a well-ordered base $\{\e_n\}_{n\in\w}$ consisting of equivalence relations such that $\e_0=\Delta_X$. In this case the formula
$$d_X(x,x')=\min\{n\in\w:(x,x')\in\e_n\}$$defines an ultrametric $d_X\colon X\times X\to\w$ generating the coarse structure of the ballean $X$. By analogy we can define an ultrametric $d_Y$ generating the coarse structure of the ballean $Y$. Since $\cov^\flat(X)=\cov^\sharp(X)=\cov^\sharp(Y)=\cov^\flat(Y)$, we can apply Theorem~1.2 of \cite{BR} (proved by the technique of towers created in \cite{BZ}) to conclude that the ultrametric spaces $X$ and $Y$ are coarsely equivalent.
\smallskip

4) $\gamma>\w$.  Since $X,Y$ are ordinal balleans of cofinality $\cof(X)=\cof(Y)=\gamma$, we can fix well-ordered bases $\{\tilde\e_\alpha\}_{\alpha<\gamma}$ and $\{\tilde\delta_\alpha\}_{\alpha<\gamma}$ of the coarse structures ${\downarrow}\E_X$ and ${\downarrow}\E_Y$, respectively.

By induction on $\alpha<\gamma$ we shall construct well-ordered sequences $\{\e_\alpha\}_{\alpha<\gamma}\subset{\downarrow}\E_X$ and $\{\delta_\alpha\}_{\alpha<\gamma}\subset{\downarrow}\E_Y$ such that for every $\alpha<\gamma$ the following conditions will be satisfied:
\begin{enumerate}
\item[(a)] $\e_\alpha=\bigcup_{\beta<\alpha}\e_\beta$ and $\delta_\alpha=\bigcup_{\beta<\alpha}\delta_\beta$ if the ordinal $\alpha$ is limit;
\item[(b)] $\e_\alpha$ and $\delta_\alpha$ are cellular entourages;
\item[(c)] $\tilde \e_\alpha\subset\e_{\alpha+1}$ and $\tilde\delta_\alpha\subset\delta_{\alpha+1}$;
\item[(d)] $\min\limits_{x\in X}\cov_{\e_\alpha}(B(x,\e_{\alpha+1})){=}\sup\limits_{x\in X}\cov_{\e_\alpha}(B(x,\e_{\alpha+1})){=}\min\limits_{y\in Y}\cov_{\delta_\alpha}(B(y,\delta_{\alpha+1})){=}
\sup\limits_{y\in Y}\cov_{\delta_\alpha}(B(y,\delta_{\alpha+1})){=}\kappa_\alpha$ for some cardinal $\kappa_\alpha$.
\end{enumerate}

We start the inductive construction by choosing cellular entourages $\e_0\in\E_X$ and $\delta_0\in\E_Y$ such that $$\sup_{x\in X}\cov_{\e_0}(B(x,\e))<\kappa\mbox{ \ and \ }\sup_{y\in Y}\cov_{\delta_0}(B(y,\delta))<\kappa$$ for any entourages $\e\in{\downarrow}\E_X$ and $\delta\in{\downarrow}\E_Y$. The existence of such entourages $\e_0$ and $\delta_0$ follows from the cellularity of $X,Y$ and the definition of the cardinals $\cov^\sharp(X)=\cov^\sharp(Y)=\kappa$.
Assume that for some ordinal $\alpha<\gamma$ and all ordinals $\beta<\alpha$ the cellular entourages $\e_\beta$ and $\delta_\beta$ have already been constructed. If the ordinal $\alpha$ is limit, then we put $\e_\alpha=\bigcup_{\beta<\alpha}\e_\beta$ and $\delta_\alpha=\bigcup_{\beta<\alpha}\delta_\beta$ and observe that the entourages $\e_\alpha$ and $\delta_\beta$ are cellular as unions of increasing chains of cellular entourages. Moreover, $\e_\alpha\in{\downarrow}\E_X$ and $\delta_\beta\in{\downarrow}\E_Y$ as $\alpha<\gamma=\add(X)=\add(Y)$.

Next, assume that the ordinal $\alpha$ is not limit and hence $\alpha=\beta+1$ for some ordinal $\beta$. Taking into account the choice of the entourages $\e_0$, $\delta_0$ and using the definitions of the cardinals $\cov^\flat(X)=\cov^\flat(Y)$, we can construct two increasing sequences of cellular entourages $\{\e'_{n}\}_{n\in\w}\subset{\downarrow}\E_X$ and $\{\delta'_{n}\}_{n\in\w}\subset{\downarrow}\E_Y$ such that
$$\sup_{x\in X}\cov_{\e'_n}(B(x,\e'_{n+1}))\le\min_{y\in Y}\cov_{\delta'_n}(B(y,\delta'_{n+1}))$$
and $$\sup_{y\in Y}\cov_{\delta'_n}(B(y,\delta'_{n+1}))\le\min_{x\in X}\cov_{\e'_{n+1}}(B(x,\e'_{n+2})).$$
The entourages $\e'_1$ and $\delta'_1$ can be chosen so that $\tilde\e_\alpha\subset\e'_1$ and $\tilde\delta_\alpha\subset\delta'_1$. Since $\add(X)=\add(Y)>\w$, the entourages $\e_{\alpha+1}=\bigcup_{n\in\w}\e'_n$ and $\delta_{\alpha+1}=\bigcup_{n\in\w}\delta'_n$ belong to the coarse structures ${\downarrow}\E_X$ and ${\downarrow}\E_Y$, respectively, and have the properties (b)--(d), required in the inductive construction.

By Lemma~\ref{key}, there are coarse equivalences $f_X\colon X\to\cbox_{\alpha<\gamma}\kappa_\alpha$ and $f_Y\colon Y\to\cbox_{\alpha<\gamma}\kappa_\alpha$. Then the multi-map $f_Y^{-1}\circ f_X\colon X\multimap Y$ is a coarse equivalence between the balleans $X$ and $Y$.\hfill $\square$


\begin{thebibliography}{}

\bibitem{BR} T.~Banakh, D.~Repov\v s, {\em Classifying homogeneous ultrametric spaces up to coarse equivalence}, Colloq. Math. (to appear); available at (http://arxiv.org/abs/1408.4818).

\bibitem{Zar} T.~Banakh, I.~Zarichnyi, {\em A coarse characterization of the Baire macro-space}, Proc. of Intern. Geometry Center. 3:4 (2010) 6--14 (available http://arxiv.org/abs/1103.5118).

\bibitem{BZ} T.~Banakh, I.~Zarichnyi, {\em Characterizing the Cantor bi-cube in asymptotic categories}, Groups Geom. Dyn. {\bf 5}:4 (2011), 691--728.

\bibitem{BDHM} N.~Brodsky, J.~Dydak, J.~Higes, A.~Mitra, {\em Dimension zero at all scales}, Topology Appl. {\bf 154}:14 (2007), 2729--2740.

\bibitem{BS} S.~Buyalo, V.~Schroeder, {\em Elements of Asymptotic Geometry}, EMS Monographs in Mathematics. (EMS), Z\"urich, 2007.

\bibitem{Nowak} P.~Nowak, G.~Yu, {\em Large Scale Geometry}, EMS Textbooks in Mathematics.  (EMS), Z\"urich, 2012.

\bibitem{b4} O.~Petrenko, I.V.~Protasov, {\em Balleans and $G$-spaces}, Ukr. Math J. {\bf 64} (2012), 344--350.

\bibitem{b5} O.~Petrenko, I.V.~Protasov, S.~Slobodianiuk, {\em Asymptotic structures of cardinals}, Appl. Gen. Topology, {\bf 15}:2 (2014), 137--146.

\bibitem{P02} I.V.~Protasov, {\em Morphisms of ball's structures of groups and graphs}, Ukr. Mat. Zh. {\bf 54}:6 (2002), 847--855.

\bibitem{P03} I.V.~Protasov, {\it Normal ball structures}, Math. Stud. {\bf 20} (2003), 3--16.

\bibitem{PB} I.~Protasov, T.~Banakh, {\em Ball stuctures and colorings of graphs and groups}, VNTL Publ. 2003, 148p.

\bibitem{PT} I.V.~Protasov, A.~Tsvietkova, {\em Decomposition of cellular balleans}, Top. Proc. {\bf 36} (2010), 77--83.

\bibitem{PZ} I.~Protasov M.~Zarichnyi, {\em General Asymptology},  Monograph Series, Vol.12, VNTL Publ., Lviv, 2007.

\bibitem{Roe} J.~Roe, {\em Lectures on Coarse Geometry}, Univ. Lecture Series, Vol.31, Amer. Math. Soc., 2003.

\end{thebibliography}
\end{document}